\documentclass[11pt,a4paper,twoside]{article}

\usepackage{amsmath}
\usepackage{amssymb}
\usepackage{amsthm}
\usepackage[dvips]{graphicx,color}
\usepackage{psfrag}

\newcommand{\pa}[1]{{\partial#1}}

\newcommand{\ga}{\alpha}

\newcommand{\gz}{\zeta}
\newcommand{\gt}{\theta}

\newcommand{\gk}{\kappa}
\newcommand{\gl}{\lambda}

\newcommand{\gs}{\sigma}

\def\XXint#1#2#3{{\setbox0=\hbox{$#1{#2#3}{\int}$}
     \vcenter{\hbox{$#2#3$}}\kern-.5\wd0}}

\theoremstyle{plain}

\newtheorem{theorem}{Theorem}[section]
\newtheorem{proposition}[theorem]{Proposition}

\newtheorem{corollary}[theorem]{Corollary}

\numberwithin{equation}{section}
\begin{document}

\title{On the Linearization of the First and Second Painlev\'e Equations}
\author{N.~Joshi\,$^{a)}$, A.~V.~Kitaev\,$^{b)}$, and P.~A.~Treharne\,$^{a)}$
\thanks{Email: N.Joshi@maths.usyd.edu.au, kitaev@pdmi.ras.ru, and P.Treharne@maths.usyd.edu.au}\\
\\
$^{a)}$School of Mathematics and Statistics, The University of Sydney, \\
New South Wales 2006, Australia, \\
$^{b)}$Steklov Mathematical Institute, Fontanka 27, St Petersburg, 191023, Russia}

\date{June 2, 2008}

\maketitle

%%%%%%%%%%%%%%%%%%%%%%%%%%%%%%%%%%%%%%%%%%%%%%%%%%%%%%%%%%%%%%%%%%%%%%%%%%%%%%%%%%%%%%%%%%%%%%%%%%%%%%%%%%%%%%%%%%%%%%%%%%%
%% ABSTRACT
%%%%%%%%%%%%%%%%%%%%%%%%%%%%%%%%%%%%%%%%%%%%%%%%%%%%%%%%%%%%%%%%%%%%%%%%%%%%%%%%%%%%%%%%%%%%%%%%%%%%%%%%%%%%%%%%%%%%%%%%%%%
\abstract
We found Fuchs--Garnier pairs in $3\times3$ matrices for the first and second Painlev\'e
equations which are linear in the spectral parameter.
As an application of our pairs for the second Painlev\'e equation we use the generalized Laplace transform
to derive an invertible integral transformation relating two its Fuchs--Garnier pairs in $2\times2$ matrices
with different singularity structures, namely, the pair due to Jimbo and Miwa and the one found by
Harnad, Tracy, and Widom. Together with the certain other transformations it allows us to relate all known
$2\times2$ matrix Fuchs--Garnier pairs for the second Painlev\'e equation with the original Garnier
pair.\vspace{24pt} \vspace{24pt}

\noindent
Short title: Linearization of Painlev\'e I and II Equations

\newpage

\setcounter{page}2
%%%%%%%%%%%%%%%%%%%%%%%%%%%%%%%%%%%%%%%%%%%%%%%%%%%%%%%%%%%%%%%%%%%%%%%%%%%%%%%%%%%%%%%%%%%%%%%%%%%%%%%%%%%%%%%%%%%%%%%%%%%
%% Introduction
%%%%%%%%%%%%%%%%%%%%%%%%%%%%%%%%%%%%%%%%%%%%%%%%%%%%%%%%%%%%%%%%%%%%%%%%%%%%%%%%%%%%%%%%%%%%%%%%%%%%%%%%%%%%%%%%%%%%%%%%%%%
\section{Introduction}
 \label{sec:intro}

The phrase ``linearization of the Painlev\'{e} equations'' is widely understood to refer to the fact that the
nonlinear Painlev\'{e} equations canbe associated with certain overdetermined systems of linear differential equations
in two complex variables.  The linear systems may be written as scalar or matrix equations and are typically referred
to as Lax pairs for the Painlev\'e equations.  In our previous paper \cite{JKT2007} we suggested calling these systems
the ``Fuchs--Garnier'' pairs for the Painlev\'e equations to pay tribute to the two scientists who first introduced
these systems in the beginning of the XX-th century.  In the period since Fuchs and Garnier first wrote their (scalar)
pairs the list has expanded so that now, for most of the Painlev\'{e} equations, there are several different
Fuchs--Garnier pairs associated with the same Painlev\'{e} equation.\footnote{Actually most of the ``new'' pairs
appeared as similarity reductions of various Lax pairs for nonlinear partial differential equations integrable via the Inverse Scattering Transform method which explains why they are often referred to as Lax pairs for the Painlev\'{e} equations.} The different Fuchs--Garnier pairs for a given Painlev\'e equation may differ from each other not only by simple gauge transformation but also by matrix dimension and/or analytic structure (the number and type of singular points). Our general belief is that all Fuchs--Garnier pairs for a given Painlev\'e equation should be equivalent in the sense that there should exist explicit transformations that map these pairs to each other. In many cases such transformations are known, however, there are still several instances of different Fuchs--Garnier
pairs whose equivalence is expected but not yet established. The main goal of this work is to construct such an explicit transformation between two Fuchs--Garnier pairs for the second Painlev\'e equation,
\begin{equation}
 \label{eq:P2}
P_2: \qquad \frac{d^{2}y}{dt^{2}} = 2y^{3} + t y + \ga
\end{equation}
where $\ga\in \mathbb{C}$ is a complex parameter. Both pairs were obtained in 1979-1980:~one pair is due to Flaschka and Newell \cite{FN1980} ($FN$-pair) and the other is due to Jimbo and Miwa \cite{JM1981II} ($JM_2$-pair). The $FN$-pair was originally obtained as a similarity reduction of the Lax pair for the modified KdV equation;  the $JM_2$-pair was originally obtained from the scalar pair of Garnier, although it can also be obtained as a similarity reduction of the Lax pair for the nonlinear Schr\"odinger equation.

The solution of the aforementioned problem associated with $P_{2}$ is intimately connected with the other central theme of this work, namely, the construction of the so-called secondary linearized Fuchs--Garnier pairs for $P_2$ and the first Painlev\'e equation,
\begin{equation}
 \label{eq:P1}
P_1: \quad \frac{d^{2}y}{dt^{2}} = 6y^{2} + t.
\end{equation}
The notion of secondary linearized Fuchs--Garnier pairs for the Painlev\'{e} equations was introduced in our previous work \cite{JKT2007} and refers to Fuchs--Garnier pairs which are \emph{linear} in the spectral parameter $\gl$, see system \eqref{eq:INT-secondary-linearization} below. However it is important to mention that the question concerning a relation between the $FN$- and  $JM_2$- pairs was one of the main motivations for both our works.

We recall that the matrix Fuchs--Garnier pairs for the Painlev\'e equations have the following form
\begin{equation}
 \label{eq:IMD}
\frac{dY}{d\gl} = \mathcal{A}(\gl,t) Y, \quad \frac{dY}{dt} = \mathcal{U}(\gl,t) Y,
\end{equation}
where $\gl \in \mathbb{C}$ is an auxiliary variable called the spectral parameter and $\mathcal{A}(\gl,t)$, $\mathcal{U}(\gl,t)\in GL(N,\mathbb C)$ are \emph{rational} functions of $\gl$ and are analytic in $t$.  Jimbo and Miwa~\cite{JM1981II} showed that for all Painlev\'e equations such pairs exist in $2\times2$ matrices.  The Frobenious compatibility condition of system \eqref{eq:IMD}
\begin{equation}
 \label{eq:compatibility}
\mathcal{A}_{t} - \mathcal{U}_{\gl} + [\mathcal{A},\mathcal{U}] = 0,
\end{equation}
where $[\,,]$ is the usual matrix commutator, being imposed identically for all values of $\gl$, is equivalent to one of the Painlev\'e equations.

As is mentioned above Fuchs--Garnier pairs can be scalar and matrix. The scalar pairs for $P_1$ and $P_2$ were obtained by Garnier~\cite{G1912}. All other scalar Fuchs--Garnier pairs that we know in the literature are related with the Garnier pairs via simple transformations. The situation with the matrix Fuchs--Garnier pairs for $P_1$ is also fairly easy to summarize.  The basic $2\times2$ matrix pair for $P_1$ ($JM_1$-pair) was found by Jimbo and Miwa~\cite{JM1981II}.  It is straightforward to prove that the $JM_{1}$-pair is equivalent to the scalar pair obtained by Garnier.  Moreover, the other $2\times2$ matrix pairs for $P_{1}$ that can be found in the literature are related with the $JM_1$-pair by the Fabri \cite{I1956} and Schlesinger transformations~\cite{JM1981II} (in the other terminology via quadratic $RS$-transformations~\cite{K2006}). As for the matrix Fuchs--Garnier pairs for $P_2$ the situation is more interesting.  To the best of our knowledge all matrix pairs for $P_2$ that were discussed so far in the literature are given in $2\times2$ matrices.  There are three different matrix pairs: the $FN$- and $JM_2$-pairs
mentioned already in the previous paragraph, and the one obtained by Harnad, Tracy, and Widom~\cite{HTW1993} ($HTW$-pair).  We give a detailed account of all these pairs in Section~\ref{sec:results}.  It is straightforward to establish that the Fabri transformation maps the $HTW$-pair into the $FN$-pair (see details in Section~\ref{sec:results}) and that the $JM_2$-pair is a matrix version of the scalar Garnier pair (see
Appendix~\ref{app:P2}). However, a direct link between the $FN$- (equivalently $HTW$-) and the $JM_2$- pairs is not that obvious; this link is one of the main matters of our paper.

Central to our investigation of this problem are the so-called secondary linearized Fuchs--Garnier pairs, the technique that we began to develop in our previous work \cite{JKT2007}.  In that work we found the secondary linearized Fuchs--Garnier pairs in $3\times3$ matrices for the third, fourth, and fifth Painlev\'e equations.  In this paper we complete the list of the secondary linearized pairs for the Painlev\'e equations by adding to it the pairs for $P_1$ and $P_2$.  The secondary linearized Fuchs--Garnier pair for the sixth Painlev\'e equation was known earlier due to Harnad~\cite{H1994} (see also Mazzocco~\cite{M2002}).

Secondary linearization is here taken to mean presenting each of the Painlev\'{e} equations in terms of a Fuchs--Garnier pair with equation on spectral parameter $\lambda$ of the following form,
\begin{equation}
 \label{eq:INT-secondary-linearization}
\big(\lambda B_1(t)+B_2(t)\big)\frac{d\Psi}{d\lambda}=\Big(\lambda B_3(t)+B_4(t)\Big)\Psi,
\end{equation}
i.e., with coefficients \emph{linear} with respect to $\gl$.  In \cite{JKT2007}, using the similarity reductions of the Lax pair for the three-wave resonant interaction (3WRI) system, we obtained secondary linearized Fuchs--Garnier pairs in $3\times3$ matrices for all the Painlev\'e equations except $P_1$ and $P_2$.  As there are no similarity reductions of the 3WRI system to $P_{1}$ and $P_{2}$ this approach could not be used to obtain secondary linearized pairs for the latter equations.  So in this paper we complete a list of the secondary linearized Fuchs--Garnier
pairs for the Painlev\'e equations.

The main advantage of the secondary linearized pairs is that the Laplace transform maps them one into another without changing the matrix dimension of the corresponding Fuchs--Garnier pairs.  This property is lost for Fuchs--Garnier pairs which are rational functions of the spectral parameter of degree more than 1.  In this work the fact that the matrix dimension is not altered is the key to constructing an explicit link between the $JM_{2}$- and the $FN$- pairs for $P_{2}$ since the two secondary linearized $3\times3$ matrix pairs that are related by the generalized Laplace transform can be reduced independently to the different $2\times2$ matrix pairs.  The reduction of our secondary linearized $3\times3$ matrix pairs to the $2\times2$ matrix Fuchs--Garnier pairs is done via two different mechanisms: (i) by a special normalization of equation~\eqref{eq:INT-secondary-linearization}; and (ii) from a degeneracy that occurs under application of the generalized Laplace transform to equation~\eqref{eq:INT-secondary-linearization}.  We call a secondary linearized Fuchs--Garnier pair \emph{degenerate} iff $\det\big(\lambda B_1(t)+B_2(t)\big)\equiv0$ for all $\gl\in\mathbb C$.  The reader will see that, in the case of $P_{1}$ and $P_{2}$, the generalized Laplace transform maps nondegenerate Fuchs--Garnier pairs into degenerate ones.

In Section~\ref{sec:P1} we present a nondegenerate secondary linearized Fuchs--Garnier pair for $P_1$ ($JKT_1$-pair).  We show that under the formal\footnote{We call a Laplace transform formal in case we do not specify its contour of integration.} Laplace transform it maps to a degenerate Fuchs--Garnier pair ($dJKT_1$-pair).  We then show that the $dJKT_{1}$-pair is equivalent to the $JM_1$-pair, which is, in turn, a matrix form of the original Garnier pair, $G_1$.  Schematically, the content of Section~\ref{sec:P1} can be described by the graph shown in
Figure~\ref{fig:P1}, where the vertices represent the corresponding Fuchs--Garnier pairs and the edges are invertible transformations relating them.
%%%%%%%%%%%%%%%%%%%%%%%%%%%%%%%%%%%%%%%%%%%%%%%%%%%%%%%%%%%%%%%%%%%%%%%%%%%%%%%%%%%%%%%%%%%%%%%%%%%%%%%%%%%%%%%%%%%%%%%%%%%
%% Figure1
%%%%%%%%%%%%%%%%%%%%%%%%%%%%%%%%%%%%%%%%%%%%%%%%%%%%%%%%%%%%%%%%%%%%%%%%%%%%%%%%%%%%%%%%%%%%%%%%%%%%%%%%%%%%%%%%%%%%%%%%%%%
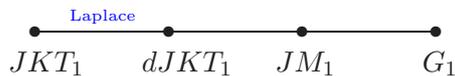
\begin{figure}[ht]
\begin{center}
\begin{picture}(200,50)
\put(25,20){\circle*{4.0}} \put(15,05){$JKT_1$} \put(25,20){\line(1,0){50}} \put(38,24){\color{blue}{\tiny Laplace}} \put(75,20){\circle*{4.0}}
\put(65,05){$dJKT_1$} \put(75,20){\line(1,0){50}} \put(125,20){\circle*{4.0}} \put(115,05){$JM_1$} \put(125,20){\line(1,0){50}}
\put(175,20){\circle*{4.0}} \put(170,05){$G_1$}
\end{picture}
\end{center}
\caption{The diagram of the Fuchs--Garnier pairs for $P_1$ and mappings between them.} \label{fig:P1}
\end{figure}
%%%%%%%%%%%%%%%%%%%%%%%%%%%%%%%%%%%%%%%%%%%%%%%%%%%%%%%%%%%%%%%%%%%%%%%%%%%%%%%%%%%%%%%%%%%%%%%%%%%%%%%%%%%%%%%%%%%%%%%%%%%

The rest of the paper is devoted to $P_2$. Schematically its content can be presented by the graph shown on Figure~\ref{fig:P2}.  The vertices and
the edges of the graph are different Fuchs--Garnier pairs for $P_2$ and the mappings between them, respectively.  The vertices abbreviated as
$JKT$ with sub- and superscripts denote $3\times3$ matrix Fuchs--Garnier pairs that we have constructed.  The subscript $\{2\}$ means the pair for
$P_2$, the prefix $d$, as above, says that the corresponding pair is degenerate, and the superscripts $\{1,2,3\}$ label different degenerate
Fuchs--Garnier pairs.  The graph is commutative and all mappings are invertible.  The edges without name correspond to the reduction
transformation from the $3\times3$ to $2\times2$ matrix and the $2\times2$ matrix to scalar Fuchs--Garnier pairs and their inverses.  Our main
result is the diagonal mapping indicated by the red edge.  It is obtained in two ways: as the composition of transformations along the upper and
lower roots connecting vertices $JM_2$ and $HTW$.  These compositions coincide which proves the commutativity of our diagram.
%%%%%%%%%%%%%%%%%%%%%%%%%%%%%%%%%%%%%%%%%%%%%%%%%%%%%%%%%%%%%%%%%%%%%%%%%%%%%%%%%%%%%%%%%%%%%%%%%%%%%%%%%%%%%%%%%%%%%%%%%%%
%% Figure2
%%%%%%%%%%%%%%%%%%%%%%%%%%%%%%%%%%%%%%%%%%%%%%%%%%%%%%%%%%%%%%%%%%%%%%%%%%%%%%%%%%%%%%%%%%%%%%%%%%%%%%%%%%%%%%%%%%%%%%%%%%%
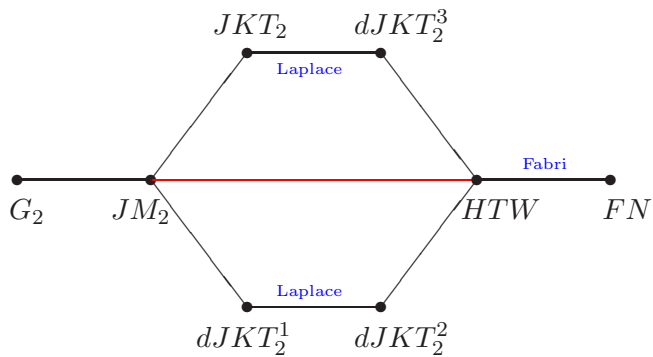
\begin{figure}[ht]
\begin{center}
\begin{picture}(250,120)
\put(22,45){$G_2$} \put(25,60){\circle*{4.0}} \put(25,60){\line(1,0){50}} \put(75,60){\circle*{4.0}} \put(75,60){\color{red}\line(1,0){120}}
\put(60,45){$JM_2$} \put(75,60){\line(3,4){35}} \put(75,60){\line(3,-4){35}} \put(111,108){\circle*{4.0}} \put(111,12){\circle*{4.0}}
\put(98,115){$JKT_2$} \put(92,-03){$dJKT_2^1$} \put(111,108){\line(1,0){50}} \put(122,100){\color{blue}{\tiny Laplace}}
\put(161,108){\circle*{4.0}} \put(111,12){\line(1,0){50}} \put(122,16){\color{blue}{\tiny Laplace}} \put(161,12){\circle*{4.0}}
\put(151,115){$dJKT_2^3$} \put(151,-03){$dJKT_2^2$} \put(161,108){\line(3,-4){35}} \put(161,12){\line(3,4){35}} \put(197,60){\circle*{4.0}}
\put(197,60){\line(1,0){50}} \put(191,45){$HTW$} \put(247,60){\circle*{4.0}} \put(244,45){$FN$} \put(214,64){\color{blue}{\tiny Fabri}}
\end{picture}
\end{center}
\caption{The commutative diagram of the Fuchs--Garnier pairs for $P_2$ and cor\-res\-pond\-ing mappings.} \label{fig:P2}
\end{figure}
%%%%%%%%%%%%%%%%%%%%%%%%%%%%%%%%%%%%%%%%%%%%%%%%%%%%%%%%%%%%%%%%%%%%%%%%%%%%%%%%%%%%%%%%%%%%%%%%%%%%%%%%%%%%%%%%%%%%%%%%%%%

This diagram is constructed in Sections~\ref{sec:results}-\ref{sec:dJKT2} and Appendices~\ref{app:P2} and \ref{app:contour}.

In Section~\ref{sec:results} we recall the main subjects of our study -- the $2\times2$ matrix Fuchs--Garnier pairs for $P_2$ due to Jimbo--Miwa
($JM_2$-pair), Flaschka--Newell ($FN$-pair), and Harnad--Tracy--Widom ($HTW$-pair).  We also show that the Fabri transformation~\cite{I1956},
which is natural to employ for the $HTW$-pair and is well-known in asymptotic theory, maps the $HTW$-pair to the $FN$-pair.  This fact was earlier
noticed in \cite{KH1999}.  Finally, we conclude this section with a presentation of the main result of this paper, i.e., we give a direct
invertible integral transformation mapping the $HTW$-pair to the $JM_2$-pair.  Appendix~\ref{app:P2} completes the general overview of the
$2\times2$ matrix pairs by showing a relation of the $JM_2$-pair to the scalar Garnier pair ($G_2$-pair).  Here we also consider one more
Fuchs--Garnier pair for $P_2$ in $2\times2$ matrices obtained by Conte and Musette~\cite{CM2000} ($CM_2$-pair) and show how it can be mapped
directly to the $JM_2$-pair without reference to the scalar $G_2$-pair.

In Section~\ref{sec:JKT2} we present two new $3\times3$ matrix secondary linearized Fuchs--Garnier pairs for $P_{2}$, one of which is
nondegenerate, together with the corresponding integral transformation between them and their reductions to the $JM_2$- and $HTW$- pairs.  Using
these results we construct a formal integral transform between the $JM_2$- and $HTW$- pairs.

In Section~\ref{sec:dJKT2} we present two other new $3\times3$ matrix secondary linearized Fuchs--Garnier pairs for $P_{2}$, different from those
in Section~\ref{sec:JKT2}.  Both pairs are degenerate and are related via the generalized Laplace transform.  As in the previous section we
construct transformations relating these pairs with the $JM_2$- and $HTW$- pairs and on this basis obtain exactly the same formal integral
transform between the latter pairs as in Section~\ref{sec:JKT2}.

Finally in Appendix~\ref{app:contour} we show how to find a contour of integration in our integral transformations which completes the proof of
our main result.

%%%%%%%%%%%%%%%%%%%%%%%%%%%%%%%%%%%%%%%%%%%%%%%%%%%%%%%%%%%%%%%%%%%%%%%%%%%%%%%%%%%%%%%%%%%%%%%%%%%%%%%%%%%%%%%
%% Section P1
%%%%%%%%%%%%%%%%%%%%%%%%%%%%%%%%%%%%%%%%%%%%%%%%%%%%%%%%%%%%%%%%%%%%%%%%%%%%%%%%%%%%%%%%%%%%%%%%%%%%%%%%%%%%%%%
\section{Fuchs--Garnier pairs for $P_1$}
 \label{sec:P1}

In Subsection~\ref{subsec:P1-pairs} we construct the Fuchs--Garnier pairs and corresponding mappings presented by the graph on
Figure~\ref{fig:P1}.  In Subsection~\ref{subsec:P1-Fabri} we discuss the Fabri-transformed $JM_1$-pairs that appeared in the literature and
present a new and simpler version of the Fabri-transformation.

\subsection{Secondary Linearization of $P_1$} \label{subsec:P1-pairs}

\begin{proposition}
 \label{prop: P1-nondegenerate}
Consider the following system of linear ODEs
\begin{equation}
 \label{sys:P1-nondegenerate}
JKT_1:\left\{
\begin{aligned}
 \frac{d\Psi}{d\mu}&=\left(\mu\begin{pmatrix}
 0&1&0\\
 0&0&1\\
 0&0&0
 \end{pmatrix}+
 \begin{pmatrix}
 -y & -z & -(4y^2+2t) \\
 0 & y & z \\
 1/4 & 0 & 0
 \end{pmatrix}\right)\Psi,\\
 \frac{d\Psi}{dt}&=2\left(\mu\begin{pmatrix}
 0&0&-1\\
 0&0&0\\
 0&0&0
 \end{pmatrix}+
 \begin{pmatrix}
 0&-y&-z\\
 -1/4&0&-2y\\
 0&-1/4&0
 \end{pmatrix}\right)\Psi,
\end{aligned}\right.
\end{equation}
Where $y=y(t)$ and $z=z(t)$ are analytic functions of $t$. Then the compatibility condition reads,
\begin{equation}
 \label{eq:P1:y-z}
\frac{dy}{dt} = z,\qquad \frac{dz}{dt}=6y^2+t,
\end{equation}
i.e. is equivalent to equation.~\eqref{eq:P1}.
\end{proposition}
\begin{proof}
The straightforward check of the Frobenious compatibility condition~\eqref{eq:compatibility} with $\lambda\to\mu$.
\end{proof}

In our terminology this is a nondegenerate secondary linearized Fuchs--Garnier pair for $P_1$.

Let us make the generalized Laplace transform of the Fuchs--Garnier pair~\eqref{sys:P1-nondegenerate} with respect to the variable $\mu$,
\begin{equation}
 \label{eq:P1-laplace}
\Psi(\mu,t) = \int_{L}e^{\lambda\mu}\Phi(\lambda,t)\,d\lambda,
\end{equation}
with such contour $L$ chosen to make the certain off-integral terms vanish.\footnote{Here we do not discuss the choice of this contour, so that we
leave this transformation at the formal level. The notation is explained in more detail in Appendix~\ref{app:contour} where the appropriate choice
of contour for the case of $P_{2}$ is given.}  The result reads:
\begin{equation}
 \label{sys:P1-degenerate}
dJKT_1:\left\{
\begin{aligned}
\begin{pmatrix}
0&1&0\\
0&0&1\\
0&0&0
\end{pmatrix}
\frac{d\Phi}{d\lambda}&=\left(-\lambda\begin{pmatrix}
 1&0&0\\
 0&1&0\\
 0&0&1
 \end{pmatrix}+
 \begin{pmatrix}
 -y & -z & -(4y^2+2t) \\
 0 & y & z \\
 1/4 & 0 & 0
 \end{pmatrix}\right)\Phi,\\
 \frac{d\Phi}{dt}&=2\left(\lambda\begin{pmatrix}
 0&-1&0\\
 0&0&0\\
 0&0&0
 \end{pmatrix}+
 \begin{pmatrix}
 0&0&0\\
 -1/4&0&-2y\\
 0&-1/4&0
 \end{pmatrix}\right)\Phi.
\end{aligned}
\right.
\end{equation}
In our terminology this is a degenerate Fuchs--Garnier pair.  To formulate our next result we recall the Fuchs--Garnier pair for $P_{1}$ found by
Jimbo and Miwa~\cite{JM1981II},
\begin{equation}
 \label{sys:P1-JM1}
JM_1:\left\{
\begin{aligned}
\frac{dY}{d\gl}&=\left(\gl^2
\begin{pmatrix}
0&1\\
0&0
\end{pmatrix}+\gl
\begin{pmatrix}
0&y\\
4&0
\end{pmatrix}
+
\begin{pmatrix}
-z&y^2+\tfrac{t}2\\
-4y&z
\end{pmatrix}\right)Y,\\
\frac{dY}{dt}&=\left(\frac{\gl}{2}
\begin{pmatrix}
0&1\\
0&0
\end{pmatrix}+
\begin{pmatrix}
0&y\\
2&0
\end{pmatrix}\right)Y.
\end{aligned}
\right.
\end{equation}

\begin{proposition}
 \label{prop:P1-degenerate}
The Fuchs--Garnier pair~{\rm\eqref{sys:P1-degenerate}} is equivalent to the $JM_1$-pair \eqref{sys:P1-JM1}.
\end{proposition}
\begin{proof}
The third row of the $\gl$-equation in system \eqref{sys:P1-degenerate} gives the following relation between the elements of the functions $\Phi$:
\begin{equation*}
\Phi_{1} = 4\gl\Phi_{3}.
\end{equation*}
Using this relation to eliminate $\Phi_{1}$ from the system~\eqref{sys:P1-degenerate} and defining $Y=\left(-\Phi_2/4,\Phi_3\right)^T$ we find
that $Y$ solves $JM_{1}$-pair~\eqref{sys:P1-JM1}.
\end{proof}

We conclude this subsection by mentioning that if the vector solution to system~\eqref{sys:P1-JM1} is written as $Y=(Y_1,Y_2)^{T}$, then the
function $V=V(\lambda,t)$ defined as $Y_2=\sqrt{\lambda-y}\,V$ satisfies the original Garnier pair for $P_{1}$,
\begin{equation*}
G_1:\left\{
\begin{aligned}
\frac{d^2V}{d\lambda^2}&=\left(\frac3{4(\lambda-y)^2}-\frac{y'}{\lambda-y}+
4\lambda^3+2t\lambda+(y')^2-4y^3-2ty\right)V,\\
\frac{dV}{dt}&=\frac1{2(\lambda-y)}\frac{dV}{d\lambda}+\frac1{4(\lambda-y)^2}V.
\end{aligned}
\right.
\end{equation*}

\subsection{The Fabri type transformation for $JM_1$-pair}
 \label{subsec:P1-Fabri}

Since the matrix coefficient of $\gl^{2}$ in the right hand side of the $JM_1$-pair~\eqref{sys:P1-JM1} has zero determinant and zero trace, it is
standard in asymptotic theory to apply the Fabri-type transformation~\cite{I1956} $\gl = \gz^{2}$ to ``cure the defect'' at infinity.  It was
first applied to system \eqref{sys:P1-JM1} by Jimbo and Miwa in the same work \cite{JM1981II} where the $JM_1$-pair was obtained (see p.437) under
the name of the ``shearing'' transformation.  As a result they obtain an equation (see (C.5) in \cite{JM1981II}) with an additional apparent
Fuchsian singularity at the origin.  Although this does not cause any problems for application of the isomonodromy deformation technique in
studying, say, asymptotics of $P_1$,\footnote{The monodromy matrix at the origin is just equal to $-I$, and thus does not depend on t.} this form
of the Fabri-transformed $JM_1$-pair does create problems in application of the Riemann-Hilbert approach.  This is due to the fact that the
corresponding connection matrix (the matrix connecting fundamental solutions at the singular points zero and infinity) for this pair now depends
on the solution of $P_1$.  This problem was first identified by Fokas, Mugan, and Zhou~\cite{FMZ}. To correct this additional problem these
authors introduced one more gauge transformation depending on a spectral parameter, and thus they produced another Fuchs--Garnier pair
($FMZ$-pair) for $P_1$.  The approach of \cite{FMZ} necessitated the introduction of an additional function, $v(t)$, which is related with the
$P_1$ function $y(t)$ via the Riccatti differential equation $iv'(t)-2v^2(t)=y(t)$.  This function appeared in the $FMZ$-pair because in the
Riemann-Hilbert setting there appears an additional parameter in the connection matrix which corresponds to the constant of integration in the
Riccatti equation.

Here we show that there exists a Fabri-type transformation which is free of both problems indicated in the previous paragraph: it does not have an
additional Fuchsian singularity at the origin like the original Fabri-transformed $JM_1$-pair and the $FMZ$-pair. Our Fabri-type transformation
for the fundamental solutions of \eqref{sys:P1-JM1} reads
\begin{equation*}
Y(\lambda,t)=
\begin{pmatrix}
1&-\zeta/2\\
0&1
\end{pmatrix}Z(\zeta,t),
\qquad \lambda=\zeta^2.
\end{equation*}
We find that our Fabri-type transformation maps the $JM_1$-pair to the following one,
\begin{equation}
 \label{sys:P1-Fabri}
JM_1/F:\left\{
\begin{aligned}
\frac{dZ}{d\zeta}=&\left(4\zeta^4
\begin{pmatrix}
1&0\\
0&-1
\end{pmatrix}
+\zeta^3
\begin{pmatrix}
0&4y\\
8&0
\end{pmatrix}+\zeta^2
\begin{pmatrix}
-4y&2z\\
0&4y
\end{pmatrix}
+\right.\\
&\zeta \left.\begin{pmatrix}
-2z&2y^2+t\\
-8y&2z
\end{pmatrix}+
\begin{pmatrix}
0&1/2\\
0&0
\end{pmatrix}
\right)Z,\\
\frac{dZ}{dt}=&\left(\zeta
\begin{pmatrix}
1&0\\
0&-1
\end{pmatrix}+
\begin{pmatrix}
0&y\\
2&0
\end{pmatrix}\right)Z.
\end{aligned}
\right.
\end{equation}

An important role in the study of the Fabri-transformed $JM_1$-pairs is played by the so-called $\sigma_1$-symmetry of the fundamental solutions
related with the reflection, $\zeta\to-\zeta$. The latter symmetry is not easy to observe looking directly at our pair~\eqref{sys:P1-Fabri},
however, a method of its derivation suggests the following identity for the fundamental solutions,
\begin{equation}
 \label{eq:P1-symmetry}
\begin{pmatrix}
1&-\zeta/2\\
0&1
\end{pmatrix}Z(\zeta)=
Y(\lambda)=
\begin{pmatrix}
1&\zeta/2\\
0&1
\end{pmatrix}Z(-\zeta)C,
\end{equation}
for some $C\in SL(2,\mathbb C)$. Thus, the $\sigma_1$-symmetry for our pair reads,
\begin{equation*}
Z(-\zeta)=
\begin{pmatrix}
1&-\zeta\\
0&1
\end{pmatrix}Z(\zeta)\tilde C,
\end{equation*}
where clearly $\tilde C\in SL(2,\mathbb C)$, and thus has nothing to do any more with the Pauli matrix $\sigma_1$!

%%%%%%%%%%%%%%%%%%%%%%%%%%%%%%%%%%%%%%%%%%%%%%%%%%%%%%%%%%%%%%%%%%%%%%%%%%%%%%%%%%%%%%%%%%%%%%%%%%%%%%%%%%%%%%%%%%%%%%%%%%%
%% P2 - 2\times2 Pairs
%%%%%%%%%%%%%%%%%%%%%%%%%%%%%%%%%%%%%%%%%%%%%%%%%%%%%%%%%%%%%%%%%%%%%%%%%%%%%%%%%%%%%%%%%%%%%%%%%%%%%%%%%%%%%%%%%%%%%%%%%%%
\section{The $2\times2$ matrix Fuchs--Garnier Pairs for $P_2$}
 \label{sec:results}
The rest of the paper is devoted to $P_2$. Henceforth the notation $y$ and $z$ means solutions of $P_2$ \eqref{eq:P2} and equation $P_{34}$ (see
equation \eqref{eq:P34}, below), respectively. Subsections~\ref{subsec:JM2}-\ref{subsec:P2-Fabri} are devoted to the review of the known results
for the $2\times2$ matrix Fuchs-Garnier pairs for $P_2$. In Subsection~\ref{subsec:main} we formulate the main result of the paper.
\subsection{The Jimbo--Miwa Pair}
 \label{subsec:JM2}
Jimbo and Miwa \cite{JM1981II} give the following matrix version of the Fuchs-Garnier pair for $P_2$:
\begin{equation}
 \label{eq:FG-JM}
JM_{2}: \left\{
\begin{aligned}
\frac{dY}{d\gl} &= \left( \gl^{2}
\begin{pmatrix}
1 & 0 \\
0 & -1
\end{pmatrix} + \gl
\begin{pmatrix}
0 & u \\
-2u^{-1}z & 0
\end{pmatrix} +
\begin{pmatrix}
z + t/2 & -u y \\
-2u^{-1}(yz + \gt) & -z - t/2
\end{pmatrix} \right) Y, \\
\frac{dY}{dt} &= \left( \frac{\gl}{2}
\begin{pmatrix}
1 & 0 \\
0 & -1
\end{pmatrix} + \frac{1}{2}
\begin{pmatrix}
0 & u \\
-2u^{-1}z & 0
\end{pmatrix} \right) Y,
\end{aligned}
\right.
\end{equation}
The compatibility condition~\eqref{eq:compatibility} for $JM_2$-pair~\eqref{eq:FG-JM} reads:
\begin{equation}
 \label{eq:IDS-JM}
\frac{du}{dt} = -yu, \quad \frac{dy}{dt} = y^{2} + z + \frac{t}{2}, \quad \frac{dz}{dt} = -2yz - \gt.
\end{equation}

Excluding the functions $u$ and $z$ from \eqref{eq:IDS-JM} we find that the function $y$ satisfies $P_2$ \eqref{eq:P2} with $\ga = \frac{1}{2} -
\gt$.

Excluding the functions $u$ and $y$ from \eqref{eq:IDS-JM} we find that the function $z$ satisfies the following second order equation,
\begin{equation}
 \label{eq:P34}
P_{34}:\qquad\frac{d^{2}z}{dt^{2}} = \frac{1}{2z} \Big(\frac{dz}{dt}\Big)^{2} - 2z^2 - tz - \frac{\gt^{2}}{2z},
\end{equation}
which, up to a scaling change of $z$ and $t$, coincides with the 34-th equation in the classical Painlev\'{e}--Gambier list, see p.~340 in
\cite{I1956}.

For the convenience of the reader we present in Appendix~\ref{app:P2} a relation of the $JM_2$-pair with the original scalar Garnier pair
($G_2$-pair) and a direct mapping of another $2\times2$ matrix version of $G_2$-pair by Conte and Musette~\cite{CM2000} ($CM2$-pair) to the
$JM_2$-pair.

\subsection{The Flaschka--Newell Pair}
 \label{subsec:FN}
Flaschka and Newell \cite{FN1980} found the following Fuchs--Garnier pair for $P_2$:
\begin{equation}
 \label{eq:FG-FN}
FN: \left\{
\begin{aligned}
\frac{dZ}{d\gz} &= \Big( -4i\gz^{2} \gs_{3} + 4y \gz \gs_{1} - 2y'\gs_{2} - i(2y^2 + t)\gs_{3} -
\frac{\ga}{\gz} \gs_{1} \Big) Z, \\
\frac{dZ}{dt} &= \Big( -i\gz\gs_{3} + y\gs_{1} \Big) Z,
\end{aligned}
\right.
\end{equation}
where $\gz$ is the spectral parameter, $t$ is the dynamical variable, prime denotes differentiation with respect to $t$, and $\gs_{1}$, $\gs_{2}$,
$\gs_{3}$ is the standard notation for the Pauli matrices:
\begin{equation*}
\gs_{1} =
\begin{pmatrix}
0 & 1 \\
1 & 0
\end{pmatrix}, \quad
\gs_{2} =
\begin{pmatrix}
0 & -i \\
i & 0
\end{pmatrix}, \quad
\gs_{3} =
\begin{pmatrix}
1 & 0 \\
0 & -1
\end{pmatrix}.
\end{equation*}
The compatibility condition~\eqref{eq:compatibility} for $FN$-pair \eqref{eq:FG-FN} implies that $y$ is a solution of $P_2$ \eqref{eq:P2}.

We note that the equation in $\gl$ in the $JM_{2}$-pair has one singularity only, an irregular singularity at infinity, while the equation in
$\gz$ in the $FN$-pair has two singularities, a regular singularity at $\gz=0$ and an irregular singularity at infinity. It follows that there
does not exist an algebraic gauge transformation for generic values of the parameter $\alpha$ between these systems.

Let us also mention one more difference between the $JM_2$- and $FN$- pairs, namely, the additional $\sigma_1$-symmetry for solutions of the
$FN$-pair,
\begin{equation}
 \label{eq:P2-sigma1-symmetry}
Z(-\zeta)=i\sigma_1Z(\zeta)C,\qquad\mathrm{where}\qquad C\in SL(2,\mathbb C).
\end{equation}

\subsection{The Harnad--Tracy--Widom Pair}
 \label{subsec:HTW}
There also exists a third $2\times2$ matrix Fuchs--Garnier pair for $P_{2}$, which was first given implicitly by Harnad, Tracy, and Widom in
\cite{HTW1993} in connection with Random Matrix Theory. Explicitly this pair was presented by Kapaev and Hubert \cite{KH1999} and in connection
with the symmetric form of $P_2$ by Noumi~\cite{N2000}. The $HTW$-pair may be written as
\begin{equation}
 \label{eq:FG-HTW}
HTW: \left\{
\begin{aligned}
\frac{dW}{d\mu} &= \left( \mu
\begin{pmatrix}
0 & 1 \\
0 & 0
\end{pmatrix} +
\begin{pmatrix}
-y & -(z + 2y^{2} + t) \\
\frac{1}{2} & y
\end{pmatrix} + \frac{1}{2\mu}
\begin{pmatrix}
\gt & 0 \\
z & -\gt
\end{pmatrix} \right) W, \\
\frac{dW}{dt} &= -\left( \mu
\begin{pmatrix}
0 & 1 \\
0 & 0
\end{pmatrix} +
\begin{pmatrix}
-y & 0 \\
\frac{1}{2} & y
\end{pmatrix} \right) W,
\end{aligned}
\right.
\end{equation}
where $\gt$ is a complex parameter.  Compatibility of system \eqref{eq:FG-HTW} implies that the functions $y$ and $z$ satisfy the following system
of nonlinear ODEs:
\begin{equation}
 \label{eq:IDS-HTW}
\frac{dy}{dt} = y^{2} + z + \frac{t}{2}, \quad \frac{dz}{dt} = -2z y - \gt.
\end{equation}
Eliminating $z$ from this system we find that the function $y$ satisfies the second Painlev\'{e} equation \eqref{eq:P2} with parameter $\ga = 1/2-
\gt$.

For completeness we recall the symmetric form of $P_2$~\cite{N2000},
\begin{equation*}
f_0'=-2qf_0+\alpha_0,\qquad f_1'=2qf_1+\alpha_1,\qquad q'=\frac12(f_1-f_0),
\end{equation*}
where $\alpha_0=1-\gt$, $\alpha_1=\gt$, and the functions $f_0=f_0(t)$, $f_1=f_1(t)$ and $q=q(t)$ in our notation read:
\begin{equation*}
f_0=z+2y^2+t,\qquad f_1=-z,\qquad q=-y,
\end{equation*}
so that the $HTW$-pair gets a natural parametrization in terms of the ``symmetric variables''.

\subsection{The Fabri transformation}
 \label{subsec:P2-Fabri}
$HTW$-pair~\eqref{eq:FG-HTW} and  $FN$-pair~\eqref{eq:FG-FN} are related via a special Fabri-type transformation \cite{KH1999}:
\begin{equation}
 \label{eq:P2-Z-HTW}
Z(\gz,t)=G(\gz)W(\mu,t),\qquad G(\gz) =\frac1{\sqrt2}
\begin{pmatrix}
1 & -1 \\
1 & 1
\end{pmatrix}
\left( \frac{i}{2\gz} \right)^{\gs_{3}/2},\qquad \mu = -2\gz^{2}
\end{equation}
The function $z$ in \eqref{eq:FG-HTW} is given by the equation $z = y' - y^{2} - \frac{t}{2}$. The $\sigma_1$-symmetry for $Z$
\eqref{eq:P2-sigma1-symmetry} follows from the Fabri transformation because of the identity,\footnote{Compare with the analogous derivation for
$P_1$ \eqref{eq:P1-symmetry}.}
\begin{equation*}
G(-\zeta)G^{-1}(\zeta)=i\sigma_1.
\end{equation*}

\subsection{Main Result}
 \label{subsec:main}
Now we are ready to formulate the main result of this paper.
\begin{theorem}
 \label{th:main}
The fundamental solution $Y(\gl,t)$ of $JM_2$-pair \eqref{eq:FG-JM} is related with the fundamental solution $W(\mu,t)$ of $HTW$-pair
\eqref{eq:FG-HTW} via the following integral transform
\begin{equation}
 \label{eq:P2-integral-transform}
\begin{pmatrix}
-u/\mu & 0 \\
0 & 2
\end{pmatrix}
\mu^{\gt/2} W(\mu,t) = \int_{L} e^{-\gl^{3}/3 + \gl(\mu - t/2)} Y(\gl,t) d\gl,
\end{equation}
where the contour $L$ is specified in Appendix~\ref{app:contour}. The inverse transformation is given by the inverse Laplace transform.
\end{theorem}

The proof of this Theorem is given in Sections~\ref{sec:JKT2} and \ref{sec:dJKT2}.

\begin{corollary}
 \label{cor:main}
The relation between the $JM_2$- and $FN$- pairs can be obtained as a composition of equations~\eqref{eq:P2-Z-HTW} and
\eqref{eq:P2-integral-transform}.
\end{corollary}

%%%%%%%%%%%%%%%%%%%%%%%%%%%%%%%%%%%%%%%%%%%%%%%%%%%%%%%%%%%%%%%%%%%%%%%%%%%%%%%%%%%%%%%%%%%%%%%%%%%%%%%%%%%%%%%%%%%%%%%%%%%
%% Section 2
%%%%%%%%%%%%%%%%%%%%%%%%%%%%%%%%%%%%%%%%%%%%%%%%%%%%%%%%%%%%%%%%%%%%%%%%%%%%%%%%%%%%%%%%%%%%%%%%%%%%%%%%%%%%%%%%%%%%%%%%%%%
\section{A $3\times3$ Fuchs--Garnier Pair for $P_{2}$}
 \label{sec:JKT2}

\begin{proposition}
 \label{proposition:R3-FG-P2}
The compatibility condition for the linear system
\begin{equation}
 \label{eq:R3-FG-P2}
JKT_{2}: \left\{
\begin{aligned}
\begin{pmatrix}
-1 & \gl + y & 0 \\
0 & -\frac{1}{2} & \gl \\
0 & 0 & -1
\end{pmatrix}
\frac{d\Phi}{d\gl} &=
\begin{pmatrix}
z + 2y^2 + t & -1 - \gk_{1} & 0 \\
-y & -\frac{1}{2}z & -1 - \gk_{2} \\
1 & 0 & 0
\end{pmatrix}
\Phi, \\
\frac{d\Phi}{dt} &=
\begin{pmatrix}
-\gl & \frac{1}{2}z & 1 + \gk_{2} \\
-1 & y & 0 \\
0 & -\frac{1}{2} & -y
\end{pmatrix}
\Phi,
\end{aligned}
\right.
\end{equation}
where $\gk_{j}, j=1,2$ are parameters, is governed by the following system of nonlinear equations:
\begin{equation}
 \label{eq:IDS-JM-prop}
\frac{dy}{dt} = y^{2} + z + \frac{t}{2}, \quad \frac{dz}{dt} = -2yz - (\gk_{1} - \gk_{2}).
\end{equation}
Eliminating $z$ from system \eqref{eq:IDS-JM-prop} we find that the function $y$ satisfies the second Painlev\'{e} equation \eqref{eq:P2} with
parameter $\ga = \frac{1}{2} - (\gk_{1} - \gk_{2})$.
\end{proposition}

\begin{proof}
The result follows from the Frobenius compatibility condition $\pa{}_{t}\pa{}_{\gl} \Phi = \pa{}_{\gl}\pa{}_{t} \Phi$.
\end{proof}

\subsection{Reduction to the $JM_{2}$-pair}

Since the compatibility condition \eqref{eq:IDS-JM-prop} depends on the parameters $\gk_{1}$ and $\gk_{2}$ only through their difference, there is
an additional degree of freedom in system \eqref{eq:R3-FG-P2}.  We will show that, by a special choice of the parameters $\gk_{1}$, $\gk_{2}$,
system \eqref{eq:R3-FG-P2} can be reduced to the $JM_{2}$-pair.

\begin{proposition}
 \label{proposition:FG-P2-JM}
If $\gk_{j} = -1$ for either $j=1$ or $j=2$ in system \eqref{eq:R3-FG-P2}, then system \eqref{eq:R3-FG-P2} can be reduced to $JM_{2}$-pair
\eqref{eq:FG-JM} plus a quadrature.
\end{proposition}

\begin{proof}
To prove this statement we first note that the coefficient matrix on the right hand side of the $\gl$ equation in \eqref{eq:R3-FG-P2} has
determinant $(1 + \gk_{1})(1 + \gk_{2})$.  Setting $\gk_{j} = -1$ for either $j=1$ or $j=2$ it follows that, upon diagonalizing the coefficient
matrix, system \eqref{eq:R3-FG-P2} can be reduced to a $2\times2$ matrix system plus a quadrature.  To simplify the following calculation we note
that system \eqref{eq:R3-FG-P2} can be written in the following form
\begin{equation*}
\frac{d\Phi}{d\gl} = -
\begin{pmatrix}
2\gl^2 + z + t & -\gl z - yz - (1 + \gk_{1}) & -2(1 + \gk_{2})(\gl + y) \\
2(\gl - y) & -z & -2(1 + \gk_{2}) \\
1 & 0 & 0
\end{pmatrix}
\Phi.
\end{equation*}
We present a proof only for a simpler case $\gk_{2} = -1$, which does not require the diagonalizing procedure.
It is the case we refer below in Subsection~\ref{subsec:4.3}. The case $\gk_1=-1$ is not employed in our work and
left to the interested reader. So, we set $\gk_{2} = -1$ and note that the third component can be solved by quadrature
once the remaining two components are determined.  The first two components of $\Phi$ satisfy the following linear
$2\times2$ matrix system:
\begin{equation}
 \label{eq:FG-P2-phi}
\begin{aligned}
\frac{d\phi}{d\gl} &= -\left( \gl^{2}
\begin{pmatrix}
2 & 0 \\
0 & 0
\end{pmatrix}
+ \gl
\begin{pmatrix}
0 & -z \\
2 & 0
\end{pmatrix}
+
\begin{pmatrix}
z + t & -yz - (1 + \gk_{1}) \\
-2y & -z
\end{pmatrix}
\right)
\phi, \\
\frac{d\phi}{dt} &= -\frac{1}{2} \left( \gl
\begin{pmatrix}
2 & 0 \\
0 & 0
\end{pmatrix}
+
\begin{pmatrix}
0 & -z \\
2 & -2y
\end{pmatrix}
\right) \phi, \quad \phi =
\begin{pmatrix}
\Phi_{1} \\
\Phi_{2}
\end{pmatrix}
.
\end{aligned}
\end{equation}
We now make a gauge transformation in system \eqref{eq:FG-P2-phi},
\begin{equation}
\phi =
\begin{pmatrix}
0 & \frac{1}{2} \\
-u^{-1} & 0
\end{pmatrix}
\chi
\end{equation}
where the function $u(t)$ is defined by $u' = -yu$.  The resulting system is given by
\begin{equation}
 \label{eq:FG-P2-chi}
\begin{aligned}
\frac{d\chi}{d\gl} &= \left( \gl^{2}
\begin{pmatrix}
0 & 0 \\
0 & -2
\end{pmatrix}
+ \gl
\begin{pmatrix}
0 & u \\
-2u^{-1}z & 0
\end{pmatrix}
+
\begin{pmatrix}
z & -uy \\
-2u^{-1}(yz + 1 + \gk_{1}) & -z - t
\end{pmatrix}
\right)
\chi, \\
\frac{d\chi}{dt} &= \frac{1}{2} \left( \gl
\begin{pmatrix}
0 & 0 \\
0 & -2
\end{pmatrix}
+
\begin{pmatrix}
0 & u \\
-2u^{-1}z & 0
\end{pmatrix}
\right) \chi.
\end{aligned}
\end{equation}
This system is gauge equivalent to the $JM_{2}$-pair given in \eqref{eq:FG-JM}. To see this, we define the parameter $\gt = (1 + \gk_{1})$ and
make the change of variables
\begin{equation*}
\chi(\gl,t) = e^{-(\gl^3/3 + \gl t/2)}Y(\gl,t).
\end{equation*}
\end{proof}

\subsection{Reduction to the $HTW$-pair}

\begin{proposition}
 \label{proposition:FG-P2-FN}
System \eqref{eq:R3-FG-P2} can be mapped to $HTW$-pair \eqref{eq:FG-HTW} by application of the generalized Laplace transform:
\begin{equation}
 \label{eq:Laplace}
\Psi(\mu,t) = \int_{L} e^{\gl\mu} \Phi(\gl,t) d\gl,
\end{equation}
where the contour $L$ is specified in Appendix~\ref{app:contour}.
\end{proposition}

\begin{proof}
We start the proof by writing system \eqref{eq:R3-FG-P2} in the following form:
\begin{equation}
 \label{eq:R3-FG-P2-linear}
\begin{aligned}
\begin{pmatrix}
-1 & \gl + y & 0 \\
0 & -\frac{1}{2} & \gl \\
0 & 0 & -1
\end{pmatrix}
\frac{d\Phi}{d\gl} &=
\begin{pmatrix}
z + 2y^2 + t & -1 - \gk_{1} & 0 \\
-y & -\frac{1}{2}z & -1 - \gk_{2} \\
1 & 0 & 0
\end{pmatrix}
\Phi, \\
\frac{d\Phi}{dt} &=
\begin{pmatrix}
0 & 0 & 1 \\
0 & 0 & 0 \\
0 & 0 & 0
\end{pmatrix}
\gl\frac{d\Phi}{d\gl} +
\begin{pmatrix}
0 & \frac{1}{2}z & 1 + \gk_{2} \\
-1 & y & 0 \\
0 & -\frac{1}{2} & -y
\end{pmatrix}
\Phi.
\end{aligned}
\end{equation}
Substituting \eqref{eq:Laplace} into \eqref{eq:R3-FG-P2-linear}, and assuming that the contour $L$ can be chosen to eliminate any remainder terms
that arise from integration-by-parts, we find
\begin{equation}
 \label{eq:R3-FG-P2-degenerate}
dJKT_{2}^{3}: \left\{
\begin{aligned}
\begin{pmatrix}
0 & 1 & 0 \\
0 & 0 & 1 \\
0 & 0 & 0
\end{pmatrix}
\mu \frac{d\Psi}{d\mu} &=
\begin{pmatrix}
\mu - (z + 2y^2 + t) & -\mu y + \gk_{1} & 0 \\
y & \frac{1}{2}\mu +\frac{1}{2}z & \gk_{2} \\
-1 & 0 & \mu
\end{pmatrix}
\Psi, \\
\frac{d\Psi}{dt} &=
\begin{pmatrix}
0 & 0 & -1 \\
0 & 0 & 0 \\
0 & 0 & 0
\end{pmatrix}
\mu \frac{d\Psi}{d\mu} +
\begin{pmatrix}
0 & \frac{1}{2}z & \gk_{2} \\
-1 & y & 0 \\
0 & -\frac{1}{2} & -y
\end{pmatrix}
\Psi.
\end{aligned}
\right.
\end{equation}
In our terminology, system \eqref{eq:R3-FG-P2-degenerate} is a degenerate secondary linearized Fuchs--Garnier pair for $P_{2}$.  The third row of
the $\mu$ equation in \eqref{eq:R3-FG-P2-degenerate} gives the following relation between the elements of the function $\Psi$:
\begin{equation*}
\Psi_{1} = \mu \Psi_{3}.
\end{equation*}
Using this relation to eliminate $\Psi_{1}$ from the above equations, we find that the remaining components of $\Psi$ satisfy the following linear
$2\times2$ matrix system
\begin{equation}
 \label{eq:FG-P2-psi}
\begin{aligned}
\frac{d\psi}{d\mu} &= \left( \mu
\begin{pmatrix}
0 & 1 \\
0 & 0
\end{pmatrix}
+
\begin{pmatrix}
-y & -(z + 2y^2 + t) \\
\frac{1}{2} & y
\end{pmatrix}
+ \frac{1}{\mu}
\begin{pmatrix}
\gk_{1} & 0 \\
\frac{1}{2}z & \gk_{2}
\end{pmatrix}
\right)
\psi, \\
\frac{d\psi}{dt} &=  -\left( \mu
\begin{pmatrix}
0 & 1 \\
0 & 0
\end{pmatrix}
+
\begin{pmatrix}
-y & 0 \\
\frac{1}{2} & y
\end{pmatrix}
\right) \psi, \quad \psi =
\begin{pmatrix}
\Psi_{2} \\
\Psi_{3}
\end{pmatrix},
\end{aligned}
\end{equation}
which is gauge equivalent to the $HTW$-pair given in \eqref{eq:FG-HTW}.
\end{proof}

\subsection{Integral transform between the $JM_{2}$- and the $HTW$- pairs}
\label{subsec:4.3}

In this section we construct \emph{explicitly} the integral transform which maps the $2\times2$ system of Jimbo--Miwa into the $2\times2$ system
of Harnad--Tracy--Widom.

\begin{theorem}
 \label{theorem:FG-P2-JM2FN}
The function $W(\mu,t)$, which solves the $HTW$-pair given in \eqref{eq:FG-HTW}, is related to the function $Y(\gl,t)$, which solves the
$JM_{2}$-pair given in \eqref{eq:FG-JM}, via the integral transform \eqref{eq:P2-integral-transform}.

\end{theorem}

\begin{proof}
{}From Proposition~\ref{proposition:FG-P2-JM} we note that the function $Y(\gl,t)$ is related to the function $\phi(\gl,t)$ in system
\eqref{eq:FG-P2-phi} via the gauge transformation
\begin{equation*}
\phi(\gl,t) =
\begin{pmatrix}
0 & \frac{1}{2} \\
-u^{-1} & 0
\end{pmatrix}
e^{-(\gl^{3}/3 + \gl t/2)} Y(\gl,t).
\end{equation*}
Similarly, if we set $\gk_{2} = -1$ and $\gk_{1} = \gt - 1$, then from Proposition~\ref{proposition:FG-P2-FN} we note that the function $W(\mu,t)$
is related to the function $\psi(\mu,t)$ in system \eqref{eq:FG-P2-psi} via the following change of variables
\begin{equation*}
\psi(\mu,t) = \mu^{-1+\gt/2} W(\mu,t).
\end{equation*}
Finally, $\psi(\mu,t)$ is related to the function $\phi(\gl,t)$ in \eqref{eq:FG-P2-phi} via the integral transform given in \eqref{eq:Laplace} and
a simple gauge transformation:
\begin{equation*}
\begin{pmatrix}
0 & \mu \\
1 & 0
\end{pmatrix}
\psi(\mu,t) = \int_{L} e^{\gl\mu} \phi(\gl,t) d\gl.
\end{equation*}
We then find
\begin{equation*}
\begin{pmatrix}
0 & \mu \\
1 & 0
\end{pmatrix}
\mu^{-1+ \gt/2} W(\mu,t) = \int_{L} e^{\gl\mu}
\begin{pmatrix}
0 & \frac{1}{2} \\
-u^{-1} & 0
\end{pmatrix}
e^{-(\gl^{3}/3 + \gl t/2)} Y(\gl,t) d\gl,
\end{equation*}
which simplifies to give \eqref{eq:P2-integral-transform}.
\end{proof}

%%%%%%%%%%%%%%%%%%%%%%%%%%%%%%%%%%%%%%%%%%%%%%%%%%%%%%%%%%%%%%%%%%%%%%%%%%%%%%%%%%%%%%%%%%%%%%%%%%%%%%%%%%%%%%%%%%%%%%%%%%%
%% SECTION Alternate
%%%%%%%%%%%%%%%%%%%%%%%%%%%%%%%%%%%%%%%%%%%%%%%%%%%%%%%%%%%%%%%%%%%%%%%%%%%%%%%%%%%%%%%%%%%%%%%%%%%%%%%%%%%%%%%%%%%%%%%%%%%
\section{Alternate secondary linearization of the $2\times2$ Fuchs--Garnier pairs for $P_{2}$}
 \label{sec:dJKT2}

In Section~\ref{sec:JKT2} we introduced a novel $3\times3$ matrix Fuchs--Garnier pair for the second Painlev\'{e} equation $P_{2}$, see system
\eqref{eq:R3-FG-P2}.  One of the principal advantages of this $3\times3$ matrix system was that it is linear with respect to the spectral
parameter $\gl$, i.e.~it is a secondary linearization of $P_{2}$.  In this Section we present an alternate secondary linearization of the
$2\times2$ matrix Fuchs--Garnier pairs for $P_{2}$.

\begin{proposition}
 \label{proposition:altR3-FG-JM}
The degenerate linear system
\begin{equation}
 \label{eq:altR3-FG-JM}
dJKT_{2}^{1}: \left\{
\begin{aligned}
\begin{pmatrix}
1 & 0 & 0 \\
0 & 1 & 0 \\
0 & 0 & 0
\end{pmatrix}
\frac{d\Phi}{d\gl} &=
\begin{pmatrix}
-(z + t) & \gl z + (yz + \gt) & 2\gl \\
2y & z & 2 \\
\gl & 0 & 1
\end{pmatrix}
\Phi, \\
\frac{d\Phi}{dt} &= \frac{1}{2}
\begin{pmatrix}
0 & z & 2 \\
-2 & 2y & 0 \\
0 & -\gl z & -2\gl
\end{pmatrix}
\Phi,
\end{aligned}
\right.
\end{equation}
is reducible to the $2\times2$ matrix Fuchs--Garnier pair for $P_{2}$ of Jimbo--Miwa.

\end{proposition}

\begin{proof}
{}From the third row in the $\gl$ equation in \eqref{eq:altR3-FG-JM} we have the following relation between elements of the function $\Phi$
\begin{equation*}
\Phi_{3} = \gl\Phi_{1}.
\end{equation*}
Using this relation to eliminate $\Phi_{3}$ from system \eqref{eq:altR3-FG-JM} we find that the remaining two components satisfy the linear
$2\times2$ matrix system given in \eqref{eq:FG-P2-phi} with $\gt = (1 + \gk_{1})$. It was shown in Proposition~\ref{proposition:FG-P2-JM} that
this system is related to $JM_2$-pair via an elementary gauge transformation.
\end{proof}

System \eqref{eq:altR3-FG-JM} can be mapped to $HTW$-pair. This fact is proved in the following Proposition.
\begin{proposition}
 \label{proposition:altFG-P2-FN}
System \eqref{eq:altR3-FG-JM} can be mapped to $HTW$-pair \eqref{eq:FG-HTW} by application of the generalized Laplace transform defined in
\eqref{eq:Laplace}.
\end{proposition}

\begin{proof}
System \eqref{eq:altR3-FG-JM} is linear with respect to the spectral variable $\gl$ and so we can immediately apply the Laplace transform in
\eqref{eq:Laplace}.  The resulting degenerate $3\times3$ matrix system is given by
\begin{equation}
 \label{eq:altR3-FG-P2-FN}
dJKT_{2}^{2}: \left\{
\begin{aligned}
\begin{pmatrix}
0 & z & 2 \\
0 & 0 & 0 \\
1 & 0 & 0
\end{pmatrix}
\frac{d\Psi}{d\mu} &=
\begin{pmatrix}
-\mu + (z + t) & -(yz + \gt) & 0 \\
-2y & -\mu - z & -2 \\
0 & 0 & -1
\end{pmatrix}
\Psi, \\
\frac{d\Psi}{dt} &= \frac{1}{2}
\begin{pmatrix}
0 & z & 2 \\
-2 & 2y & 0 \\
\mu - (z + t) & yz + \gt & 0
\end{pmatrix}
\Psi.
\end{aligned}
\right.
\end{equation}
In order to simplify the following calculation, we make a gauge transformation in system \eqref{eq:altR3-FG-P2-FN} of the form
\begin{equation}
 \label{eq:altR3-FG-FN-gauge1}
\Psi =
\begin{pmatrix}
0 & 0 & 1 \\
z^{-1} & z^{-1} & 0 \\
0 & -\frac{1}{2} & 0
\end{pmatrix}
\chi.
\end{equation}
The resulting system is given by
\begin{subequations}
 \label{eq:altR3-FG-P2-FN2}
\begin{align}
 \label{eq:altR3-FG-P2-FN2-mu}
\begin{pmatrix}
1 & 0 & 0 \\
0 & 0 & 0 \\
0 & 0 & 1
\end{pmatrix}
\frac{d\chi}{d\mu} &=
\begin{pmatrix}
-(y + \gt z^{-1}) & -(y + \gt z^{-1}) & -\mu + (z + t) \\
-\mu - z & -\mu & -2yz \\
0 & \frac{1}{2} & 0
\end{pmatrix}
\chi, \\
 \label{eq:altR3-FG-P2-FN2-t}
\frac{d\chi}{dt} &=
\begin{pmatrix}
0 & 0 & \mu - 2z - t \\
-y - \gt z^{-1} & -y - \gt z^{-1} & -\mu + z + t \\
\frac{1}{2} & 0 & 0
\end{pmatrix}
\chi.
\end{align}
\end{subequations}
The second row in equation \eqref{eq:altR3-FG-P2-FN2-mu} implies a relation between the elements of the function $\chi$:
\begin{equation}
 \label{eq:Psi-relation}
\chi_{2} = -\chi_{1} - \frac{1}{\mu} (z\chi_{1} + 2yz\chi_{3}).
\end{equation}
Using this relation to eliminate $\chi_{2}$ from system \eqref{eq:altR3-FG-P2-FN2} we find that the remaining two components satisfy the following
linear $2\times2$ matrix system
\begin{equation}
 \label{eq:FG-P2-psi-tilde}
\begin{aligned}
\frac{d\psi}{d\mu} &= \left( \mu
\begin{pmatrix}
0 & -1 \\
0 & 0
\end{pmatrix}
+
\begin{pmatrix}
0 & z + t \\
-\frac{1}{2} & 0
\end{pmatrix}
+ \frac{1}{\mu}
\begin{pmatrix}
yz + \gt & 2yz(y + \gt z^{-1}) \\
-\frac{1}{2}z & -yz
\end{pmatrix}
\right)
\psi, \\
\frac{d\psi}{dt} &=  \left( \mu
\begin{pmatrix}
0 & 1 \\
0 & 0
\end{pmatrix}
+
\begin{pmatrix}
0 & -2z - t \\
\frac{1}{2} & 0
\end{pmatrix}
\right) \psi, \quad \psi =
\begin{pmatrix}
\chi_{1} \\
\chi_{2}
\end{pmatrix}.
\end{aligned}
\end{equation}
Making the gauge transformation
\begin{equation}
 \label{eq:altR3-FG-FN-gauge2}
\psi =
\begin{pmatrix}
-1 & -2y \\
0 & 1
\end{pmatrix}
\mu^{\gt/2} W,
\end{equation}
in system \eqref{eq:FG-P2-psi-tilde} we get system \eqref{eq:FG-HTW}.
\end{proof}

\subsection{Alternate proof of Theorem~\ref{theorem:FG-P2-JM2FN}}

\begin{proof}
{}From Proposition~\ref{proposition:altR3-FG-JM} we note that the degenerate system \eqref{eq:altR3-FG-JM} is reducible to $JM_{2}$-pair
\eqref{eq:FG-JM}, while from Proposition~\ref{proposition:altFG-P2-FN} the degenerate system \eqref{eq:altR3-FG-P2-FN2} is reducible to $HTW$-pair
\eqref{eq:FG-HTW}.  The function $\Phi(\gl,t)$ in \eqref{eq:altR3-FG-JM} is related to the function $\chi(\mu,t)$ in \eqref{eq:altR3-FG-P2-FN2}
via the following integral transform
\begin{equation*}
\begin{pmatrix}
0 & 0 & 1 \\
z^{-1} & z^{-1} & 0 \\
0 & -\frac{1}{2} & 0
\end{pmatrix}
\chi(\mu,t) = \int_{L} e^{\gl\mu} \Phi(\gl,t) d\gl.
\end{equation*}
The first two components of this expression give
\begin{equation*}
\begin{pmatrix}
\chi_{3}(\mu,t) \\
z^{-1} \big(\chi_{1}(\mu,t) + \chi_{2}(\mu,t)\big)
\end{pmatrix}
= \int_{L} e^{\gl\mu}
\begin{pmatrix}
\Phi_{1}(\gl,t) \\
\Phi_{2}(\gl,t)
\end{pmatrix}
d\gl.
\end{equation*}
Using relation \eqref{eq:Psi-relation} to eliminate $\chi_{2}$ from this expression we find
\begin{equation*}
\begin{pmatrix}
0 & 1 \\
-\frac{1}{\mu} & -\frac{2y}{\mu}
\end{pmatrix}
\psi(\mu,t) = \int_{L} e^{\gl\mu} \phi(\gl,t) d\gl,
\end{equation*}
where $\psi = (\chi_{1},\chi_{3})^{T}$ and $\phi = (\Phi_{1},\Phi_{2})^{T}$.  From Proposition~\ref{proposition:altR3-FG-JM}, the function
$Y(\gl,t)$ is related to the function $\phi(\gl,t)$ via the gauge transformation
\begin{equation*}
\phi(\gl,t) =
\begin{pmatrix}
0 & \frac{1}{2} \\
-u^{-1} & 0
\end{pmatrix}
e^{-(\gl^{3}/3 + \gl t/2)} Y(\gl,t).
\end{equation*}
Similarly, from Proposition~\ref{proposition:altFG-P2-FN}, the function $W(\mu,t)$ is related to the function $\psi(\gl,t)$ in system
\eqref{eq:FG-P2-psi-tilde} via the gauge transform given in \eqref{eq:altR3-FG-FN-gauge2}.  Combining these two expressions we find
\begin{equation*}
\begin{pmatrix}
0 & 1 \\
-\frac{1}{\mu} & -\frac{2y}{\mu}
\end{pmatrix}
\begin{pmatrix}
-1 & -2y \\
0 & 1
\end{pmatrix}
\mu^{\gt/2} W(\mu,t) = \int_{L} e^{\gl\mu}
\begin{pmatrix}
0 & \frac{1}{2} \\
-u^{-1} & 0
\end{pmatrix}
e^{-(\gl^{3}/3 + \gl t/2)} Y(\gl,t) d\gl,
\end{equation*}
which simplifies to give \eqref{eq:P2-integral-transform}.
\end{proof}

\appendix

%%%%%%%%%%%%%%%%%%%%%%%%%%%%%%%%%%%%%%%%%%%%%%%%%%%%%%%%%%%%%%%%%%%%%%%%%%%%%%%%%%%%%%%%%%%%%%%%%%%%%%%%%%%%%%%%%%%%%%%%%%%
%% Appendix A
%%%%%%%%%%%%%%%%%%%%%%%%%%%%%%%%%%%%%%%%%%%%%%%%%%%%%%%%%%%%%%%%%%%%%%%%%%%%%%%%%%%%%%%%%%%%%%%%%%%%%%%%%%%%%%%%%%%%%%%%%%%
\section{On the matrix versions of the Garnier pair for $P_2$}
 \label{app:P2}
There are two matrix versions of the original scalar $G_2$-pair: the $JM_2$-pair and the Fuchs-Garnier pair obtained by Conte and
Musette~\cite{CM2000}, the $CM_2$-pair. Since both pairs were obtained by some simple transformations from the $G_2$-pair they are equivalent and
should be related via a simple gauge transformation. However, in \cite{JM1981II} there are no explicit details given of the relation between the
$JM_2$- and $G_2$- pairs and so, to complete our diagram in Figure~\ref{fig:P2}, we give this relation here. In the work \cite{CM2000} one finds
details of the derivation of the matrix $CM_2$-pair from the scalar $G_2$-pair, however it looks very much different from the matrix $JM_2$-pair.
To establish their equivalence we present a direct transformation between the $JM_2$- and $CM_2$- pairs avoiding the original ``intermediate''
object $G_2$-pair.

We begin with the relation between the $JM_2$- and $G_2$- pairs. Consider any column of the fundamental matrix solution to system
\eqref{eq:FG-JM}, $Y_k =(Y_{k1},Y_{k2})^{T}$, $k=1,2$, then the function $V = V(\gl,t)$ defined as $Y_{k1} = \sqrt{u(\gl - y)} \, V$, for any $k$,
satisfies the original Garnier pair for $P_{2}$,
\begin{equation}
 \label{eq:G2-pair}
G2: \left\{
\begin{aligned}
&\frac{d^2V}{d\gl^2}=\Big(\frac3{4(\gl-y)^2}-\frac{y'}{\gl-y}+(y')^2+\gl^4-y^4+t(\gl^2-y^2)+2\ga(\gl-y)\Big)V,\\
&\frac{dV}{dt}=\frac1{2(\gl-y)}\frac{dV}{d\gl}+\frac1{4(\gl-y)^2} V,
\end{aligned}
\right.
\end{equation}
where $\ga=\frac12-\gt$.

Thus we have shown that the $JM_{2}$-pair gives a matrix representation of the scalar $G_{2}$-pair. Of course, different matrix representations
are also possible and in particular we discussed above the $CM_2$-pair. Clearly, all these matrix representations are equivalent, however
sometimes a direct explicit mapping between them is not immediately obvious. The $CM_2$-pair reads:
\begin{equation}
 \label{eq:CM2-pair}
CM_{2}: \left\{
\begin{aligned}
\frac{dM}{d\gl}&=\left(\gl^3
\begin{pmatrix}
0&1\\
0&0
\end{pmatrix}+
\gl^2
\begin{pmatrix}
0&y\\
0&0
\end{pmatrix}+
\gl
\begin{pmatrix}
0&y^2+t\\
1&0
\end{pmatrix}+
\begin{pmatrix}
-y'&y^3+ty+2\ga\\
-y&y'
\end{pmatrix}
\right)M,\\
\frac{dM}{dt}&=\left(\frac{\gl^2}2
\begin{pmatrix}
0&1\\
0&0
\end{pmatrix}+\gl
\begin{pmatrix}
0&y\\
0&0
\end{pmatrix}+
\begin{pmatrix}
0&(3y^2+t)/2\\
1/2&0
\end{pmatrix}
\right)M,
\end{aligned}
\right.
\end{equation}
where $y=y(t)$ is any solution of $P_2$ \eqref{eq:P2} and $y'=dy/dt$.

To map this system directly into the $JM_{2}$-pair given in \eqref{eq:FG-JM} we introduce the parameter $\gt$ and functions $z=z(t)$, $u=u(t)$ as
follows:
\begin{equation*}
\gt=\frac12-\ga,\qquad z=y'-y^2-\frac{t}2,\qquad \frac{du}{dt}=-yu.
\end{equation*}
Then a relation between the fundamental solutions of \eqref{eq:FG-JM}, $Y=Y(\gl,t)$, and \eqref{eq:CM2-pair}, $M=M(\gl,t)$, is given by the
following gauge transformation depending on the spectral parameter:
\begin{equation}
 \label{eq:CM2-gauge}
M(\lambda,t)=G(\gl,t)Y(\lambda,t),\qquad G(\gl,t)=
\begin{pmatrix}
\gl+y&1\\
1&0
\end{pmatrix}
\begin{pmatrix}
u^{-1/2}&0\\
0&u^{1/2}
\end{pmatrix}.
\end{equation}

%%%%%%%%%%%%%%%%%%%%%%%%%%%%%%%%%%%%%%%%%%%%%%%%%%%%%%%%%%%%%%%%%%%%%%%%%%%%%%%%%%%%%%%%%%%%%%%%%%%%%%%%%%%%%%%%%%%%%%%%%%%
%% Appendix B
%%%%%%%%%%%%%%%%%%%%%%%%%%%%%%%%%%%%%%%%%%%%%%%%%%%%%%%%%%%%%%%%%%%%%%%%%%%%%%%%%%%%%%%%%%%%%%%%%%%%%%%%%%%%%%%%%%%%%%%%%%%
\section{Contour of integration in the generalized Laplace\\ transform}
 \label{app:contour}
Here we explain how to define the contour of integration in equation~\eqref{eq:P2-integral-transform}. In an analogous way the reader can find the
contour of integration in the Laplace transform for the Fuchs--Garnier pairs for $P_1$ considered in Section~\ref{sec:P1} as well as contours of
integration for the Laplace transforms of the Fuchs-Garnier pairs for some other Painlev\'e equations considered in our previous
work~\cite{JKT2007}.

We start with some remarks on notation.  Let us denote the columns of the matrix $X$ by $X^k=X^k$, $k=1,2$, i.e., $X=(X^1,X^2)$.  We assume that
$X(\gl)$ is an integrable function of $\gl$.  The notation $L$ in equation \eqref{eq:P2-integral-transform}, as well as in the Laplace
transform~\eqref{eq:P1-laplace} considered in Section~\ref{sec:P1}, is understood to mean a pair of contours of integration for each column $X^k$:
\begin{equation}
 \label{eq:B-L}
L=(\mathcal{L}^1,\mathcal{L}^2),\quad\mathrm{so\;that}\quad \int_LX(\lambda)d\lambda=
\left(\int_{\mathcal{L}^1}X^1(\lambda)d\lambda,\int_{\mathcal{L}^2}X^2(\lambda)d\lambda\right).
\end{equation}
For brevity we call $L$ simply the contour of integration.  Assume that, for $k=1,2$, the $2\times2$ matrices $C_k$ and the $2\times2$ diagonal
matrices $D_k$ are independent of $\lambda$.  Then for any two integrable $2\times2$ matrix functions $X_k(\lambda)$ we have the following
identity:
\begin{equation*}
\int_L(C_1X_1(\lambda)D_1+C_2X_2(\lambda)D_2)d\lambda=C_1\int_LX_1(\lambda)d\lambda\,D_1+ C_2\int_LX_2(\lambda)d\lambda\,D_2.
\end{equation*}

Let us now commence our analysis with a discussion of some general issues related with the choice of the contour $L$ in
\eqref{eq:P2-integral-transform}. On one hand, we need a closed contour of integration because in Sections~\ref{sec:JKT2} and \ref{sec:dJKT2} we
assume vanishing of the certain off-integral terms that appear due to the integration by parts. On the other hand, the integrand in
\eqref{eq:P2-integral-transform} is an entire function of $\lambda$, thus the contour of integration cannot be closed in the finite domain of the
complex $\lambda$ plane and should pass through the point at infinity, since we would like to get a nontrivial function $W(\mu,t)$.  This brings
us immediately to the issues of convergence and conditions on $L$ which ensure that the integral does not vanish.

To cope with these two problems we have to consider in more detail the asymptotic behaviour of the fundamental solutions of the $JM_2$-pair. The
main instruments for this are the canonical solutions $Y_n(\lambda,t)$, $n\in\mathbb Z$ of $JM_2$-pair~\eqref{eq:FG-JM}, which are defined
(uniquely) by their asymptotic expansion as $\lambda\to\infty$,
\begin{equation}
 \label{eq:JM2-asymptotics}
Y_n(\gl,t)\sim\left(I+\mathcal{O}\Big(\frac1{\gl}\Big)\right)
\exp\left(\left(\gl^{3}/3+\gl t/2-\gt\log\gl\right)\gs_3\right),
\end{equation}
where $I$ is the identity matrix and $\log\lambda=\log|\lambda|+i\arg\lambda$, in the corresponding sectors
\begin{equation*}
\mathcal{S}_n=\left\{\gl\,:\,\frac{\pi}6+\frac{\pi(n-2)}3<\arg\lambda<\frac{\pi}6+\frac{\pi n}3\right\}.
\end{equation*}
The canonical solutions are related with each other by the Stokes matrices, $S_n$:
\begin{equation*}
Y_{n+1}(\gl)=Y_n(\gl)S_n,\qquad S_{2k+1}=
\begin{pmatrix}
1&s_{2k+1}\\
0&1
\end{pmatrix},\qquad
S_{2k}=
\begin{pmatrix}
1&0\\
s_{2k}&1
\end{pmatrix},
\end{equation*}
where the numbers $s_n\in\mathbb C$ are called the Stokes multipliers. Moreover,
\begin{equation*}
Y_{n+6}\big(\gl e^{2\pi i}\big)=Y_n(\gl)e^{-2\pi i\theta\sigma_3},\qquad n\in\mathbb Z.
\end{equation*}
>From the above definitions one can easily deduce that the Stokes matrices satisfy the so-called cyclic relation, although this is not important in
the following.  We now find the leading term of asymptotics of the integrand in \eqref{eq:P2-integral-transform} assuming that $Y$ coincides with
$Y_n$,
\begin{equation}
 \label{eq:integrand-Y}
\hat Y_n\equiv e^{-\gl^{3}/3+(\mu-t/2)\gl}\, Y_n(\gl,t) \underset{\substack{\lambda\to\infty \\
\lambda\in{\mathbb S}_n}}{\sim}
\begin{pmatrix}
\lambda^{-\theta}e^{\lambda\mu}&0\\
0&\lambda^{\theta}e^{\lambda\mu}e^{-2\gl^{3}/3-t\gl}
\end{pmatrix}.
\end{equation}
In the above asymptotics we have two exponential functions, namely, $e^{\lambda\mu}$ and $e^{-2\gl^{3}/3-t\gl}$. The contour of integration should
be chosen such that both exponents should vanish as $|\lambda|\to\infty$.  The ``Laplace exponent'', $e^{\lambda\mu}$, we make small as
$|\lambda|\to\infty$ by demanding that $\mathrm{Re}\,\lambda\mu<0$. Since the contour has two directions as $|\lambda|\to\infty$ we get two
conditions on $\mu$. To satisfy both conditions we must have the angle between these directions less than $\pi$. The directions themselves are
chosen such that the second exponent, $e^{-2\gl^{3}/3-t\gl}$, decays. It can be any directions inside of the following sectors:
$\mathcal{S}_{2k}\cap\mathcal{S}_{2k+1}$, $k\in\mathbb Z$. To achieve a better convergence the contour can be chosen asymptotic to the rays
$\{\lambda: \arg\lambda=2\pi n/3, n\in\mathbb Z\}$.

We have to take care that integrating a fundamental solution $Y(\gl,t)$ in equation~\eqref{eq:P2-integral-transform} we arrive at some fundamental
solution $W(\mu,t)$. This means that, between the asymptotic directions of our contour, an arch of circle centered at the origin and having a
large radius should cross a Stokes ray where the corresponding column of the fundamental solution $Y(\gl,t)$ is affected by the Stokes phenomenon.
If this condition is not satisfied then the column vector will be exponentially vanishing on that circle (as the radius enlarges to infinity) and,
by the Cauchy theorem, the the integral of that column taken along the contour also vanishes.  In fact, this condition actually means that the
above mentioned arch intersects {\bf two} Stokes rays of the {\bf matrix} solution.  In the sector between these Stokes rays our column is
unbounded as $|\lambda|\to\infty$, so that the Cauchy theorem does not apply.  Strictly speaking, after we choose the contour we also have to
prove that we obtain the fundamental solution $W(\mu,t)$ by, say, calculating its asymptotics as $\mu\to\infty$\footnote{This calculation should
be obvious for the experienced reader, we omit it here.}.

Now we apply the above principle to construct the contour $L$ in \eqref{eq:P2-integral-transform}. Consider, for
example, the canonical solution $Y_{2k}$, $k\in\mathbb Z$. The reader can check that defining contour $L=L_k$
with $\mathcal{L}^1_k=\mathcal{L}^2_k =\mathcal{L}_k$, where $\mathcal{L}_k$ is defined as any smooth
simple\footnote{The absence of selfintersections is not actually important.}
curve asymptotic to the rays $R^\varepsilon_{2k+1}$ and $R^\varepsilon_{2k+2}$, where
\begin{equation*}
\mathcal{R}^\varepsilon_n:=\{\lambda:\;\arg\lambda=-\frac\pi6+\frac{\pi n}3+(-1)^n\varepsilon\},
\qquad
n\in\mathbb Z\qquad\mathrm{and}\qquad
0<\varepsilon<\pi/3,
\end{equation*}
fits all the conditions indicated in the above paragraphs, provided the following condition is imposed on $\mu$:
\begin{equation}
 \label{eq:B-mu-condition}
\frac{\pi}3-\frac{2\pi k}3+\varepsilon<\arg\mu<\pi-\frac{2\pi k}3-\varepsilon.
\end{equation}
This condition comes from the imposing the exponential decay condition, $\pi/2<\arg\lambda\mu<3\pi/2$, for
asymptotics of $\hat Y_{2k}$ on the rays $R^\varepsilon_{2k+1}$ and $R^\varepsilon_{2k+2}$. Note that the
asymptotics on $R^\varepsilon_{2k+1}$ is given by equation~\eqref{eq:integrand-Y} for $n=2k$, while on
$\mathcal{R}^\varepsilon_{2k+2}$ the
Stokes phenomenon dictates the following leading term of asymptotics for $\hat Y_{2k}$,
\begin{equation*}
\hat Y_{2k} \underset{\gl \to \infty}{\sim}
\begin{pmatrix}
\lambda^{-\theta}e^{\lambda\mu} & 0\\
0 & \lambda^{\theta}e^{\lambda\mu}e^{-2\gl^{3}/3-t\gl}
\end{pmatrix}
\begin{pmatrix}
1 + s_{2k}s_{2k+1} & -s_{2k+1} \\
-s_{2k} & 1
\end{pmatrix}.
\end{equation*}

With this choice of the contour $L$ in \eqref{eq:P2-integral-transform} we construct the function $W(\mu)$
for all values of $\mu$ except the rays, $\mu:\;\arg\mu\neq\pi(1+2k)/3, k\in\mathbb Z$. On the latter rays
$W(\mu)$ can be obtained via the analytic continuation. We can, also, obtain $W(\mu)$ on these rays by a proper
choice of the contour $L$. In the latter case it splits into the two contours \eqref{eq:B-L}.

For example, assume $\arg\mu=\pi$. Construct the following solution $Y=(Y_2^2,Y_4^2)$, where $Y_2^2$ and $Y_4^2$
are the second columns of the canonical solutions $Y_2$ and $Y_4$, respectively. The function $Y$ is a fundamental
solution of system~\eqref{eq:FG-JM}, iff $s_3\neq0$. Consider now the following
anti-Stokes rays\footnote{It is the rays that originally were called the Stokes rays.}
\begin{equation*}
\mathcal{R}_n:=\{\lambda:\;\arg\lambda_0=2n\pi/3\},\qquad n=0,1,2.
\end{equation*}
Now we can define contour $L=(\mathcal{L}_1,\mathcal{L}_2)$ in \eqref{eq:P2-integral-transform} where
for $k=1,2$, $\mathcal{L}_k$ is any smooth simple\footnote{The absence of selfintersections is not actually important.}
curve with two asymptotic rays $R_k$ and $R_0$.
%%%%%%%%%%%%%%%%%%%%%%%%%%%%%%%%%%%%%%%%%%%%%%%%%%%%%%%%%%%%%%%%%%%%%%%%%%%%%%%%%%%%%%%%%%%%%%%%%%%%%%%%%%%%%%%%%%%%%%%%%%%
%% BIBLIOGRAPHY
%%%%%%%%%%%%%%%%%%%%%%%%%%%%%%%%%%%%%%%%%%%%%%%%%%%%%%%%%%%%%%%%%%%%%%%%%%%%%%%%%%%%%%%%%%%%%%%%%%%%%%%%%%%%%%%%%%%%%%%%%%%

\end{document}